\documentclass[12pt]{article}
 \usepackage[dvipsnames]{xcolor}
 \colorlet{Mycolor0}{gray!10!white!90!}
 \colorlet{Mycolor1}{gray!50!white!50!}
 \colorlet{Mycolor2}{gray!50!white!40!yellow!10!}
 \usepackage{tikz}
 \usetikzlibrary{shapes,arrows}
 \usepackage{pgfplots}
\usepackage[english]{babel}

\usepackage[utf8]{inputenc}
\usepackage{abraces}

\usepackage{mathrsfs}
\usepackage{amsfonts}
\usepackage{amsthm}
\usepackage{latexsym,amsmath,amssymb}
\usepackage{graphicx,color}
\usepackage{subcaption}
\usepackage{authblk}

\newcommand{\ds}{\displaystyle}

\usepackage[all]{xy}

\newtheorem{thm}{Theorem}[section]
\newtheorem{prop}[thm]{Proposition}

\newtheorem{corol}[thm]{Corollary}

\newtheoremstyle{obs}
  {3pt}
  {3pt}
  {}
  {}
  {\bfseries}
  {.}
  {.5em}
  {}
\theoremstyle{obs}
\newtheorem{remark}[thm]{Remark}
\newtheorem{defn}[thm]{Definition}
\newtheorem{example}[thm]{Example}
\newtheorem{Ass}[thm]{Assumption}

\def\qed{\ifvmode\removelastskip\fi
{\unskip\nobreak\hfil\penalty50\hbox{}\nobreak\hfil \hbox{\vrule
height1.2ex width1.2ex}\parfillskip=0pt \finalhyphendemerits=0
\par \smallskip}}

\newcommand{\R}{\mathbb{R}}

\title{{\bf Global controllability tests for geometric hybrid control systems}}
\author[1]{M. Barbero Li\~n\'an}
\author[3]{J. Cort\'es}
\author[2]{D. Mart\'in de Diego}
\author[3]{S. Mart\'inez}
\author[4]{M. C. Mu\~noz Lecanda}

\affil[1]{\small Departamento de Matem\'atica Aplicada, Universidad Polit\'ecnica de Madrid,   Av. Juan de Herrera 4, 28040 Madrid, Spain}
\affil[2]{\small Instituto de Ciencias Matem\'aticas (CSIC-UAM-UC3M-UCM), C/Nicol\'as Cabrera 13-15, 28049 Madrid, Spain}
\affil[3]{\small Department of Mechanical and Aerospace Engineering, University of California, San Diego, USA}
\affil[4]{\small Department of Mathematics, Universidad Polit\'ecnica de Catalu\~na, Edificio C--3, Campus Norte, C/Jordi Girona 1, 08034, Barcelona, Spain}

\date{\today}

\begin{document}

\maketitle

\begin{abstract} 
	This paper introduces a novel geometric framework to define and study hybrid systems. We exploit the geometry and topology of the set of jump points, where the instantaneous change of dynamics takes place, in order to gain controllability for the system. This approach allows us to describe new global controllability tests for hybrid control systems. We illustrate these results with several examples where none of the continuous control systems are controllable, but the associated hybrid system is controllable because of the characteristics of the jump set.

\end{abstract}

\section{Introduction}\label{Sec:intro}

A dynamical system is described by a set of differential equations. When the equations depend on controls we can actuate over the system to obtain a specific objective. However, many systems around us cannot be described by a simple dynamical system, for instance, car transmission, thermostats, bipedal walkers~\cite{AmesBipedal}, electric vehicles~\cite{ElectricVehicle}, multifingered robots~\cite{Multifinger}, etc. Those systems can only be described by  using a family of dynamical systems allowing instantaneous changes (called jumps) in the dynamics among them under some particular conditions. These systems are known as (control) hybrid systems and they consist of an interaction among different continuous control systems through instantaneous jumps.

Recently, mathematicians~\cite{Brogliato,GoebelBook,Lerman,
Liberzon,HybridOpt,VanDerSchaft} have focused on
the description of hybrid systems in order to bring more understanding to all the possible case studies. The difficulty in handling these systems is given by the interaction among all the continuous dynamical systems plus all the possible instantaneous jumps. 
This can be described by a directed graph as the one in Figure~\ref{Fig:HS}, where each node represents a different dynamics (denoted by the small letter) and the edges determine the feasible transitions.

\begin{figure}[h]
	\centering \includegraphics[scale=0.9]{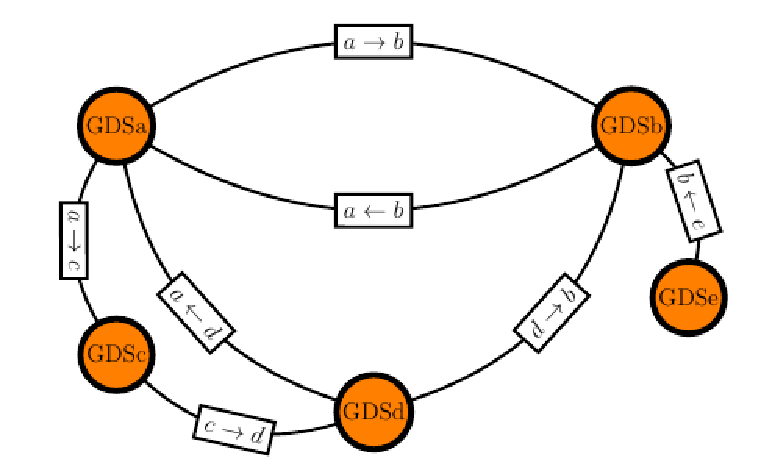}
	\caption{Directed graph associated with a hybrid system. Nodes represent dynamics modes called Generalized Dynamical Systems (GDS), whereas the edges represent transitions between them.} \label{Fig:HS}
\end{figure}

For control systems, it is important to discuss their accessibility and controllability properties~\cite{2005BulloLewisBook,NijmeijerVanDerSchaft}. These properties are related with the topology of the reachable set, that is, the set of points that can be reached by admissible trajectories. Global controllability means that there is an admissible trajectory between any two points in the configuration manifold. When instantaneous jumps appear, it is necessary to guarantee that any two discrete states can be joined through the directed graph associated with the hybrid system. That leads to the notion of discrete controllability and properties from graph theory are needed to study it. Once the discrete controllability has been checked, the controllability of the global system has to be studied. In the literature, there are some results on the topic under strong assumptions, such as at any point of the configuration manifold the system can change from one dynamics to another~\cite{2002BulloZefran}. However, that assumption allows to study the system almost as a mechanical system with many more control input vector fields, the ones coming from all the possible dynamics. That is why we have decided to geometrize the notion of hybrid control systems in order to study global controllability
having in mind the geometry and the topology of the set of jump points. We show in this paper that the instantaneous jumps can contribute to achieve global controllability under some assumptions. That idea was mentioned in~\cite{2007YangBlanke}, but only applied to study control linear systems such as $\dot{x}=Ax+Bu$ and from an algorithmic viewpoint. In this paper we look into the geometry of the set of jump points and the leaves described by the nonlinear control systems, such as $\dot{x}=f(x,u)$ where $x$ denotes the positions and $u$ the controls, to obtain global controllability results for hybrid control systems.

Setting the foundations to describe geometrically hybrid systems has the ultimate goal of providing a framework suitable to extend geometric methods used to solve tracking problems, optimal control problems, geometric integration for mechanical control systems to the hybrid control world.

The paper is organized as follows: We first introduce the notion of generalized dynamical system in Section~\ref{Sec:GenDynSystem} to introduce the new geometric framework to describe hybrid systems in Section~\ref{Sec:GenHybSyst}. The particular case of hybrid control systems with controls in the continuous dynamics is described in Definition~\ref{Def:GHCS}. Section~\ref{Sec:ControlTestHCS} contains the novel  controllability tests for hybrid control systems. 
We provide algorithmic
and geometric results,
and rewrite them infinitesimally for control-linear and control-affine systems using the Orbit Theorem~\cite{AndrewNotes,S73}. Examples are provided along the text to highlight how the jump map can contribute to gain controllability of the total system when each of the continuous dynamics independently is not controllable. Appendix~\ref{App:HS} reviews the known definitions of hybrid systems used in the literature~\cite{GoebelBook,
VanDerSchaft} to establish the analogies and similarities with the geometric description of hybrid systems given in this paper.

\section{Generalized dynamical system}\label{Sec:GenDynSystem}

For the description of a hybrid system,  we first need to introduce the constitutive pieces that we call generalized dynamical systems. These systems cover many cases of interest in geometric mechanics and control theory. We believe that they are the most suitable element to geometrize hybrid systems as explained in the sequel.

A \textit{generalized dynamical system} is characterized by a fiber bundle $\tau_E: E\rightarrow M$ equipped with a differentiable map $\rho: E\rightarrow TM$ and a submanifold $D$ of $E$ (possibly with corners)
  that projects onto $M$ by $\tau_D\colon=\tau_{E|D}=\tau_M\circ \rho_{|D}\colon D \rightarrow M$, where $\tau_M: TM\rightarrow M$ is the canonical projection of the tangent bundle.
 We write the generalized dynamical system as the quadruple $(E, M, \rho, D)$.

A curve $\gamma: I\subseteq {\mathbb R}\rightarrow D$ is a \textit{solution of the generalized dynamical system} $(E, M, \rho, D)$ if
 \[
\frac{d\sigma}{dt}= \rho(\gamma(t))\, ,
 \]
 where $\sigma:  I\subseteq {\mathbb R}\rightarrow M$ is the projection by $\tau_E$ of $\gamma$, that is, $\sigma=\tau_E\circ \gamma$.

Now, we give a description in coordinates to understand better the meaning of  a generalized dynamical system. In coordinates $(x^i)$ on $M$ and fiber coordinates $(x^i, y^{\alpha})$ on $E$, we have the following local expressions:  $\tau_E(x^i, y^{\alpha})=(x^i)$ and
$\rho(x^i, y^\alpha)=(x^i, f^i(x, y))$. Additionally assume that $D$ is defined by some inequality constraint functions $\Phi^l(x, y)\geq 0$.  Locally, a solution curve $\gamma(t)=(x^i(t), y^{\alpha}(t))$ satisfies the following system of differential equations subject to inequality constraints:
\begin{equation}\label{adf}
\left\{\begin{array}{rcl}
\displaystyle\frac{dx^i}{dt}(t)&=&f^i(x(t), y(t))\, ,\\
0&\leq&\Phi^l(x(t),y(t))\, .
\end{array}\right.
\end{equation}
If the inequality is an equality,~\eqref{adf} is a system of differential-algebraic equations (DAEs). 
Such systems appear for instance when considering nonholonomic mechanics, port-Hamiltonian systems as the one in Example~\ref{Ex:PHsyst} and systems with impact, as the simple bouncing ball in Example~\ref{Ex:rebote}. Moreover, Example~\ref{Ex:AffineLie} has some velocities restricted by the affine subspace described by the drift and input vector fields. The second set of constraints in~\eqref{adf} also appears in Example~\ref{Ex:control} for the controls, in Examples~\ref{Ex:theory},~\ref{Ex:2HCS} and in  Section~\ref{SubS:Ex} for the controls and states. 

We provide here some illustrative examples to make clear the notion of generalized dynamical systems.

\begin{example}{\bf Integral curves of vector fields.}\label{adgh}
{\rm
A typical dynamical system is given by a vector field $X$ on $M$, $X\in{\mathfrak X}(M)$. In this case $E=M$,  $D=M$, $\rho=X: M\rightarrow TM$ and the generalized dynamical system is described by $(M,M, X,M)$.
The integral curves of the vector field fulfills the differential equation:
\[
\frac{dx}{dt}=X(x)\, .
\]

}
\end{example}

\begin{example}{\bf Nonlinear control systems.}\label{Ex:control}
Consider $E=M\times  {\mathbb R^k}$ where $\tau_E=\hbox{pr}_1: E=M\times  {\mathbb R^k}\rightarrow M$ is the projection onto the first factor. The space $E$ plays the role of the control bundle.
The generalized dynamical system is given by $(M\times \R^k, M, \rho, M\times U)$ where $U\subseteq {\mathbb R^k}$ is the set of admissible controls and $\rho:
M\times \R^k\rightarrow TM$ describes the control equations, $\rho(x, u)=(x, f(x, u))$. The equations of motion are given by
\[
\frac{dx}{dt}=f(x,u), \quad u \in U\, .
\]
Typically, the control set contains the origin and is closed and bounded. Depending on the nature of the control set, some additional constraints could appear in~\eqref{adf}. For instance, if $U=[a_1,b_1]\times \cdots \times [a_k,b_k]$, then the system~\eqref{adf} will include the inequality constraints $a_s\leq u_s\leq b_s$ for $s=1,\dots,k$.

\end{example}

\begin{example}\label{Ex:AffineLie} {\rm
		{\bf Affine left-invariant  control systems on a Lie group.}
Let $G$ be a Lie group and ${\mathfrak g}$ its Lie algebra. Given an affine subspace 
\[
\Gamma=\{\xi_0+\sum_{i=1}^m u_i \xi_i\; |\; \xi_i\in {\mathfrak g}, i=0, \ldots, m\}
\]
Then the generalized dynamical system  is described  by
$(G\times {\mathfrak g} , G, \rho, G\times \Gamma)$ and 
$\rho (g, \xi)=(g, g\xi)$ and the equations of motion are
\[
\frac{dg}{dt}=g (\xi_0+\sum_{i=1}^m u_i \xi_i)
\]
Observe that the  controllability properties of the system are studied in terms of $D= G\times \Gamma$. This is one of the reasons for describing the structure of the generalized dynamical systems by four components, that is,  $(E, M, \rho, D)$.}
\end{example}

	\begin{example}\label{Ex:PHsyst} {\rm
			{\bf Port-Hamiltonian systems.}
			The study  of port-Hamiltonian systems provides an adequate  framework for the  description of network models of physical systems.
			These systems are described by the following set of equations: 
			\begin{equation}\label{port}
			\begin{split}
			\dot{x}&=\left[J(x)-R(x)\right]\frac{\partial H}{\partial x}+g(x)u\\
			y&=g^T(x)\frac{\partial H}{\partial x}
			\end{split}
			\end{equation}
			where the matrix $J$ is skew-symmetric, the matrix $R$ satisfies $R(x)=R^T(x)\geq 0$, $g$ is the input matrix and $u$ and  $y$ are the input and output variables. 
			This type of equations also admits a representation as a generalized dynamical system, where now $(x, u)$ are  coordinates for $D$, $(x,u, \dot{x}, y)$ are coordinates for $E$ and $x$ coordinates for $M$. 
			Here, $D$ is represented as a  submanifold of $E$ given by the constraints represented by Equations~\eqref{port} and 
			$\rho(x, u)=(x, 	\dot{x})=\left[J(x)-R(x)\right]\frac{\partial H}{\partial x}+g(x)u)$.
		}
	\end{example}

\begin{example} \label{el-eq} {\rm
{\bf Euler-Lagrange equations   \cite{AMBook}.}
For a regular Lagrangian function $L: TQ\rightarrow {\mathbb R}$ we know that the solutions to Euler-Lagrange equations are the integral curves of a vector field $\xi_L\in {\mathfrak X}(TQ)$
which is a second-order differential equation. In canonical coordinates $(q^i, v^i)$ on $TQ$:
\[
\xi_L=v^i\frac{\partial }{\partial q^i}+F_L^i(q, v)\frac{\partial}{\partial v^i}
\]
where
\[
F_L^i(q, v)=W^{ij}\left(\frac{\partial L}{\partial q^j}-\frac{\partial^2 L}{\partial v^j\partial{q}^k}v^k \right)\, ,
\]
$(W^{ij})$ is the inverse matrix of the hessian matrix $\left(W_{ij}=\frac{\partial^2 L}{\partial v^i\partial v^j}\right)$. Therefore, this case is a generalized dynamical system as the ones considered in  Example \ref{adgh}, where now $M=TQ$ and $\rho=\xi_L: M\rightarrow TM$, that is, $(TQ, TQ, \xi_L, TQ)$. The equations of motion are:
\begin{equation*}\label{adf-1}
\left\{\begin{array}{rcl}
\displaystyle\frac{dq}{dt}&=&v\, ,\\ \\
\displaystyle\frac{dv}{dt}&=&F_L(q, v)\, ,
\end{array}\right.
\end{equation*}
which are equivalent to the classical Euler-Lagrange equations for a regular Lagrangian:
\[
\frac{d}{dt}\left(\frac{\partial L}{\partial v^i}\right)-\frac{\partial L}{\partial q^i}=0\; .
\]
}
\end{example}

\section{Geometric hybrid systems}\label{Sec:GenHybSyst}

The adjective hybrid indicates mixed character of an object. The hybrid systems include different generalized dynamical systems and relationships among them by means of transitions from one particular generalized dynamical system to another. We build on the models for hybrid systems proposed in~\cite{GoebelBook,Lerman,
Liberzon,VanDerSchaft} to provide a geometric framework to reason about hybrid systems. Those previous frameworks in the literature have been summarized in Appendix~\ref{App:HS} to show that they are included in the geometric framework described here. 
Our approach allows us  a clearer analysis of the controllability of a hybrid system and include in a natural and geometric way nonholonomic mechanical systems defined on manifolds in the hybrid world, together with ODE systems. This opens the way to bring into the picture geometric techniques for the reduction under symmetry and the discretization of hybrid systems.

Here we extend the geometric framework considered in~\cite{VanDerSchaft} to describe hybrid systems using our notion of generalized dynamical system. As described below, set--valued maps are needed for introducing hybrid systems. More information on those maps can be found, for instance, in~\cite{BookSetValued}.

\begin{defn}\label{Def:GHS} A \textit{geometric hybrid system} is a six tuple $\Sigma=(A,E,M,\rho,D,R)$ where
\begin{itemize}
\item $A$ is a finite set of discrete modes associated with the different generalized dynamical  systems, that is, there are as many discrete modes as different continuous dynamics the system has.  Each discrete mode is denoted by $a$.
\item For each $a\in A$ there is a generalized dynamical system $\Sigma_a=(E_a,M_a,\rho_a,D_a)$ obtained from the following elements:
\begin{itemize}
\item $E$ is a global space where all the objects of the system are well-defined.
$E$ is a fiber bundle over a manifold $M$ with projection $\tau_E: E\rightarrow M$. Moreover,   both spaces
 $E$ and $M$ are fibered over $A$ with projection $\varphi_E: E\rightarrow A$ and  $\varphi_M: M\rightarrow A$ satisfying
 $\varphi_E=\varphi_M\circ \tau_E$.
The fibers of $\varphi_E$ and $\varphi_M$ are denoted by $E_a=\varphi_E^{-1}(a)$ and $M_a=\varphi_M^{-1}(a)$, respectively. For each $a\in A$ we describe  $M_a$ by coordinates $(x^{i_a})$ and $E_a$ by fibered coordinates
$(x^{i_a},y^{\alpha_a})$.
\item We have a map $\rho: E\rightarrow TM$ such that $\varphi_E=\varphi_M\circ \tau_M\circ \rho$, where $\tau_M\colon TM\rightarrow M$ is the canonical projection of the tangent bundle. Thus, for each $a$ we have a well defined map $\rho_a: E_a\rightarrow TM_a$ with local expression $\rho(x^{i_a},y^{\alpha_a})=(x^{i_a}, f^{i_a}_{(a)} (x, y))$.
\item $D$ defines the continuous dynamics as a submanifold (possible with corners and boundaries) of $E$, it also  fibers onto $A$ with projection $\varphi_D\colon D \rightarrow A$. Locally, we assume that the fiber  $D_a=\varphi_D^{-1}(a)$ is described by a set of inequality constraints $\Phi_{(a)}^{l_a}(x^{i_a}, y^{\alpha_a})\geq 0$.
\end{itemize}
\item $R\colon D\rightrightarrows D$ is a set--valued map called jump map such that every point in $D$ where a change of dynamics can take place is sent to
a group of points in the image of $D$ called successors. If a point in $D$ is not a jump point, then $R$ is not defined. That transition or jump could involve a change of the discrete mode and/or the initial condition for the dynamics associated with the next discrete mode because the hybrid trajectories are not assumed to be continuous, as described later. The graph of $R$, ${\rm Graph}\, R$, is a subset of $D\times D$ that carries  the information of the set of points  where jumps take place and all the corresponding successors.
\end{itemize}
\end{defn}

Locally, for each $a\in A$ the corresponding generalized dynamical system has equations of motion analogous to the ones in Equation \eqref{adf}:
\begin{equation}\label{adf_1}
\left\{\begin{array}{rcl}
\displaystyle\frac{dx^{i_a}}{dt}&=&f^{i_a}_{(a)}(x, y)\, , \\
0&\leq&\Phi^{l_a}_{(a)}(x,y)\, .
\end{array}\right.
\end{equation}

  All the information needed to define a geometric hybrid system is contained in the following diagram:
\begin{equation}\xymatrix{&& S \ar@{^{(}->}[rrd]   && && && TM \ar[d]^{\txt{\small{$\tau_{M}$}}} \\ {\rm Graph}\, R\ar[urr]^{\txt{\small{${\rm pr}_2$}}}
	\ar[drr]^{\txt{\small{${\rm pr}_1$}}} && && D \ar@{^{(}->}[rr] \ar[drr]_{\txt{\small{$\varphi_{D}$}}} && E \ar[urr]^{\txt{\small{$\rho$}}}  \ar[rr]^{\txt{\small{$\tau_{E}$}}} \ar[d]^{\txt{\small{$\varphi_{E}$}}} && M  \ar[dll]^{\txt{\small{$\varphi_{M}$}}} \\   && B  \ar@{^{(}->}[urr] &&  && A &&
}\label{Diagram:GHS} \end{equation}
where ${\rm pr}_i\colon {\rm Graph}\, R\subset D\times D \rightarrow D$ is the projection onto the $i$th factor and
\begin{equation}\label{Def:SB}
S=\{d \in D\, : \,\exists\; \widetilde{d}\in D \mbox{ such that } d=R(\widetilde{d})\} \; , \quad B=\{\widetilde{d}\in D\, : \,\exists\; d\in D \mbox{ such that } d=R(\widetilde{d})\}\, ,
\end{equation}
are the set of successors and the set of jump points, respectively. These two sets can also be described as $S={\rm pr}_2({\rm Graph}\, R )$ and $B={\rm pr}_1({\rm Graph}\, R)$.

As soon as the trajectory hits $B$, the jump map $R$ is active indicating the possibility of a change of dynamics, that is, prescribing the possible points from which the trajectory continues to evolve according to a different continuous dynamics. 

In order to handle better geometric hybrid systems, it is convenient to introduce a few more notions  to be used in the sequel. Having in mind the elements introduced in Definition~\ref{Def:GHS}, we have:
\begin{itemize}
\item As $B$ fibers onto $A$, the \textit{set of jump points} in the mode $a$ is described by $B_a= B\cap  D_a$.
\item The set of all possible successors  from a jump point $d_a\in B_a$ that fibers onto the discrete mode $a$ is given by $R(d_a)$.  The set of successors coming from mode $a$ is described by $S_a=S\cap R(B_a)$ and it fibers onto $A$ indicating all the possible discrete modes that can be reached from $A$.
\item The  \textit{set of feasible transitions} between two discrete modes $G\subseteq A\times A$ is defined by $G=\{(a,b)\in A\times A \, | \, \exists \; d_a,\; d_b\in D \mbox{ such that } (d_a,d_b)\in {\rm Graph}\, R\}$.
\end{itemize}

\begin{Ass} \label{AssImp} {\rm It is crucial for the results of the paper to assume that all the fibers of $B$ over the points $\tau_E(B_a)$ always have maximum dimension for every discrete mode $a$. In other words, once a jump point in the base manifold is achieved, the jump map is defined for the entire manifold, though in some cases the jump might imply no changes in those values. }
\end{Ass}

 It is clear that a hybrid system can be interpreted as a directed graph whose  set of nodes 
 specifies the different generalized dynamical systems and they are connected by directed edges that determine the possible transitions
  from a dynamical system to another as appears in Figure~\ref{Fig:HS}.1. Self-edges could also appear in the graph and it will imply a spontaneous change of the initial condition, as for instance happens in the following example.

\begin{example}{\bf Lagrangian hybrid systems: the bouncing ball.} \label{Ex:rebote}
{\rm
Consider a regular  Lagrangian system $L: TQ\rightarrow {\mathbb R}$. 
Suppose that we have a  function $h: Q\rightarrow {\mathbb R}$ determining unilateral constraints, i.e., the set of admissible configurations is defined by $h(q)\geq 0$. When the ball hits the floor or a surface described by $h(q)=0$ with a particular velocity, it will bounce back in a specific way.
In this case
$E=M=TQ$, $\rho=\xi_L: TQ\rightarrow TTQ$ and
\[
D=\{v_q\in TQ\; \mid \; h(q)\geq 0\}\; .
\]
 This example is a hybrid system with one node $(TQ, TQ, \xi_L, D)$ that corresponds with a generalized dynamical system and one self-edge.
The map $R: D \rightrightarrows D$ describes a jump condition (elastic, inelastic...) that only takes place in
 \[
 B=\{v_q\in TQ\; \mid\;  h(q)=0 \hbox{  and  } \langle dh(q),	v_{q}\rangle< 0\},\]
where the inequality implies that the velocity at the impact point cannot be zero, neither tangent to the surface of impact. The successors are given by
 \[
 S=\{R(v_q)\in TQ\;\mid\;  v_q\in B\}\, .
 \]
In the case of the bouncing ball always moving perpendicular to the floor, $R(v_q)=-v_q$ and the jump is active whenever the height is zero and the velocity points toward the floor. For more general impacts, the reflection laws must be considered. This example is also described in~\cite[Example 1.1]{GoebelBook} and it is essential to model collisions  between bodies, bipedal robots~\cite{AmesBipedal,MadImpulsive}.
}
\end{example}

\subsection{Hybrid trajectories}

Once the geometric hybrid systems have been introduced, let us describe a solution to the systems. The hybrid trajectories must carry the information of both the continuous and discrete states. The hybrid nature might lead to discontinuities in the trajectory.

\begin{defn}A \textit{solution to the geometric hybrid system or hybrid trajectory with initial discrete mode $a$ and initial condition $d_a\in D_a=\varphi_D^{-1}(a)$} is a piecewise absolutely continuous curve $\gamma \colon [0,T] \subseteq \mathbb{R} \rightarrow D$ such that
\begin{enumerate}
\item  $\gamma(0)=d_a$;
\item there exist $t_0,\dots , t_{N+1}$ such that $0=t_0< t_1 < t_2 < \dots <t_{N+1}=T$
and an associated  sequence $a_0,\dots , a_{N}$ with $(a_{j-1}, a_{j})\in G$ for each $1\leq j\leq N$ ;
\item for each $j=0,\dots, N$ there exist $\gamma_j\colon [t_j,t_{j+1}]\rightarrow D$ such that ${\rm Im}\, \gamma_j \subseteq D$ and $\varphi_D\circ  \gamma_j=a_j$ and \begin{equation}\label{Eq:GHSsol} \dfrac{\rm d}{{\rm d} t} \, (\tau_E \circ \gamma_j)(t)=(\rho \circ \gamma_j)(t) \quad \mbox{for almost every time } t\in [t_j,t_{j+1}];
\end{equation}
\item $\gamma_{|[t_j,t_{j+1})}=\gamma_{j}$ and $\gamma(T)=\gamma_{N}(T)$;
\item $(\gamma_j(t_{j+1}),\gamma_{j+1}(t_{j+1})) \in {\rm Graph}\, R$ for each $t=0,\dots, N-1$.
\end{enumerate}
\label{Defn:HybridTraj}
\end{defn}

\remark The hybrid trajectory $\gamma$ at each time carries the information of the discrete mode and the point in $D$ since $D$ fibers onto $A$.  The submanifold $D$ could be identified with the state manifold, positions and velocities, control bundle, etc, depending on the systems under consideration as shown in the Examples in Section~\ref{Sec:GenDynSystem}.

\remark Condition 4 indicates that the solution to the geometric hybrid system is not necessarily continuous from the left at times where the jump takes place, unless the set--valued jump map $R$ leads to that. See Example~\ref{Ex:rebote} where the point in $M$ is fixed, but the velocity changes as a consequence of the bouncing effect. In that case the trajectory is continuous on $M$, but not in $D$. However, in Example~\ref{Ex:Theoretic} the trajectory could be also discontinuous on $M$.

\remark
In the given definition discrete transition times are not admitted, that is, the trajectory cannot immediately jump after getting to a new discrete mode, but it will instead have to
temporarily evolve according to the continuous dynamics associated with that new discrete mode. These discrete transition times are discussed in~\cite{VanDerSchaft}, but when considering the hybrid trajectory in $M_a$ only the initial time and the final time in the sequence of the discrete transition times are considered to take the states. That is why we have decided not to admit the discrete transition times in this work. 

In the literature of hybrid systems the trajectories often are defined algorithmically~\cite{Liberzon,VanDerSchaft}. In the geometric framework we have introduced here and having in mind that the trajectory changes to a different discrete mode every time it hits $B$, the algorithm to obtain trajectories includes the following steps as illustrated in Figure~\ref{Fig:HSJumps}:
\begin{enumerate}
\item \textbf{Starting point:} Let $d\in D$ such that $a=\varphi_D(d) \in A$.
\item  The trajectory $\gamma$ evolves in $D_a$ as long as it is not in $B_a$.
\item If $\gamma(t)\in B_a$, choose $\tilde{d}\in D$, if exists, such that $(\gamma(t),\tilde{d})\in {\rm Graph}\, R$.
\begin{enumerate}
\item If $\tilde{d}$ is not a jump point,
that is, it is not in $B_{\varphi_D(\tilde{d})}$, return to 2 by taking $d:=\tilde{d}$, $a:=\varphi_D(\tilde{d})$.
\item If $\tilde{d}$ is in $B$, return to 3.
 \end{enumerate}
\item If $\gamma(t)\in B_a$ and for any $\tilde{d}$ such that $(\gamma(t),\tilde{d})\in {\rm Graph}\, R$ and $\tilde{d}\notin D_{\varphi_D({\tilde{d}})}$, the trajectory finishes at time $t$.
\end{enumerate}

\begin{figure}[h] 
	\begin{center}
		\resizebox{16cm}{!}{
			
			\fbox{\begin{tikzpicture}
				\draw[fill=Mycolor0] plot [smooth cycle] coordinates {(5,0.25) (6,0.35) (6.5, 0.2) (7,.5) (7,1.65)  (7,2.75) (4.8,2.8) } node at (4.2,1.7) {$D_{a_{j+1}}$};
				
				\draw[fill=Mycolor0] plot [smooth cycle] coordinates {(0.8,.1)(1.5,.2)(2.8,.5)(2.9,1.5)(2.8,2.8)(1.4,2.5)(0.5,0.5)} node at (0.6,1.8) { $D_{a_j}$};
				\draw[fill=Mycolor1] plot [smooth cycle] coordinates {(0.9,.4)(1.5,.4)(1.5,1)(0.9,1)} node at (0.65,0.4) {\small$\scriptscriptstyle B_j$};
				\draw[fill=Mycolor2] plot [smooth cycle] coordinates {(2.8,1.7)(2,1.7)(2,2.2)(2.78,2.2)} node at (2.2,2.4) {\small$ \scriptscriptstyle S_{j-1}\cap D_{a_j}$};
				\draw[fill=Mycolor2]  plot [smooth cycle] coordinates {(5.7,.4)(6.4,.45)(6.4,1)(6,1)} node at (5.6,1.2) {$\scriptscriptstyle S_{j}\cap D_{a_{j+1}}$};
				\draw[fill=Mycolor1]  plot [smooth cycle] coordinates {(4.9,2)(4.9,2.7)(6.2,2.8)(6.4,2.4)} node at (6.77,2.3) {$\scriptscriptstyle B_{j+1}$};
				\node[label=below: $\scriptscriptstyle\gamma(t_{j-1})$]  at (2.42,2.4) {};
				\draw(2.7,1.8) node[circle, inner sep=0.8pt, fill=red]  {};
				\draw(6.1,0.7) node[circle, inner sep=0.8pt, fill=red]  {};
				\node[label=below:$\scriptscriptstyle\gamma_j(t_j)$]  at  (1.2,0.92) {};
				\node[label=below:$\scriptscriptstyle\gamma(t_j)$]  at (6.1,0.95) {};
				\node[label=below:$\scriptscriptstyle\gamma_{j+1}(t_{j+1})$]  at (5.55,2.8) {};
				\draw[->] [red] plot [smooth, tension=1] coordinates {  (2.7,1.78) (2.3,1) (1.5,1)  (1.2,0.7)};
				\draw[->] [red] plot [smooth, tension=1.5] coordinates {  (6.1,0.7) (6.4,1.5) (6.25,2.5)};
				
				\path[->] (1.65,0.6) edge[bend right] node[above]{$\scriptscriptstyle R_{a_{j-1}a_{j}}$} (5.6,0.5);
				\path[->] (0.4,3) edge[bend left] node[above]{$\scriptscriptstyle R_{a_{j}a_{j+1}}$} (2,2.6);
				\path[->] (6,2.85) edge[bend left] node[above]{$\scriptscriptstyle R_{a_{j+1}a_{j+2}}$} (7,3.25);
				\end{tikzpicture} }}
		\caption{Representation of a hybrid trajectry when evolving from mode $a_j$ to $a_{j+1}$.}
		\label{Fig:HSJumps}
		
	\end{center}
\end{figure}
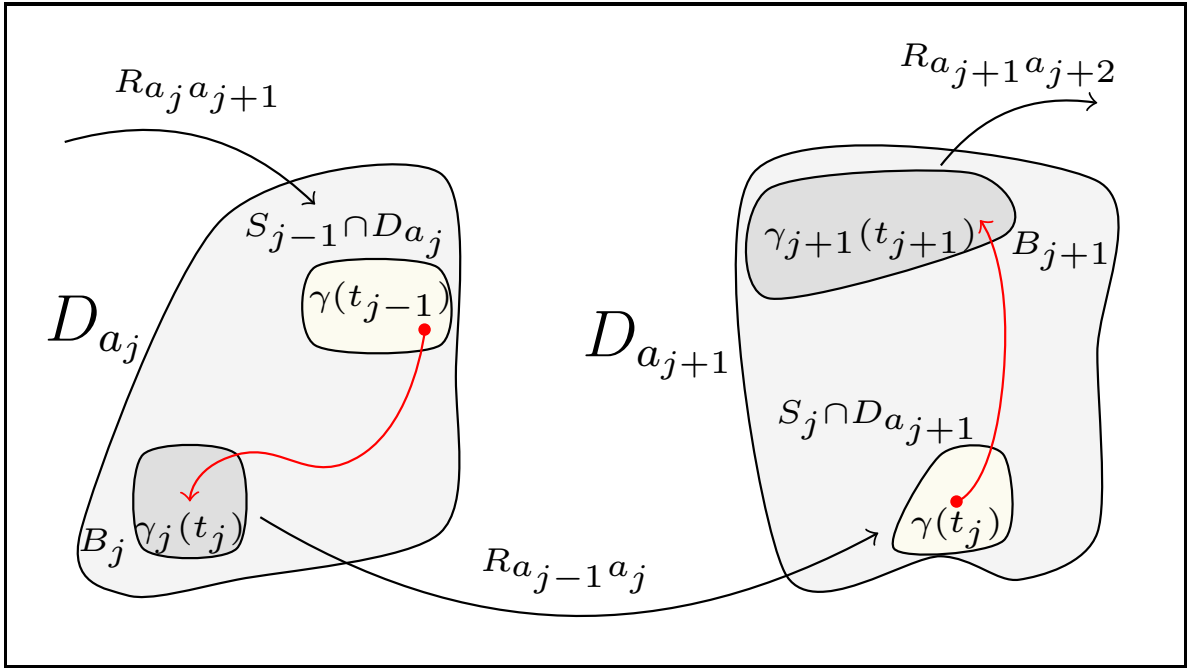

Note that if the trajectory never reaches $B$, it will evolve depending on the imposed conditions such as fixed final time, end-point conditions, the state constraints, etc, as a simple generalized dynamical system. After choosing $\tilde{d}$ such that $(\gamma(t),\tilde{d})\in  {\rm Graph}\, R$, the trajectory will evolve in $D_{\varphi_D(\tilde{d})}$ if $\tilde{d}$ is not in $B_{\varphi_D(\tilde{d})}$. As shows the step 4, it could happen that
when solving the system~\eqref{adf_1} for the initial point $\tilde{d}$, the inequality constraints are not satisfied. Then a dead end has been reached and the trajectory cannot continue. For instance,
\begin{equation}\label{eqDeadEnd}
\left\{ \begin{array}{l} \dot{x}=-1\, , \\
x\geq 0\, , \\
x(t_0)=0\, .
\end{array}\right.
\end{equation}

\example  Let us consider the following example from~\cite[Example 4.1]{2001CLMM}
that consists of a particle in the plane subjected to constraints on a half-space and moving freely in the other half. This example can be interpreted as a
geometric hybrid system by considering $(A= \{a,b\},E=A\times T\mathbb{R}^2, M=A\times\mathbb{R}^2,\rho={\rm id},D,R) $ where
\begin{equation*}
D=\{(a,(x,y,v_x,v_y))\in E\, | \, x\leq 0\}\, \bigcup   \, \{(b,(x,y,v_x,v_y))\in E\, |\, x\geq 0, \; v_x=v_y\},
\end{equation*}
and the jump map
is given by
\begin{eqnarray*}
R\colon B &\rightrightarrows & D\\
(a,(0,y,v_x,v_y)) & \longmapsto & \left\{ \begin{array}{lcl} \left(b, \left(0,y, \dfrac{v_x+v_y}{2},  \dfrac{v_x+v_y}{2}\right)\right)  &  \mbox{   if   } & v_x+v_y\geq  0,\\
\left(b, \left(0,y, v_x,v_y\right)\right)  &  \mbox{   if   } & v_x+v_y<0,  \end{array}\right. \\
(b,(0,y,v_x,v_y)) & \longmapsto &  (a, (0,y,v_x,v_y))\, .
\end{eqnarray*}
Thus,
\begin{equation*}
B= \{(a,(0,y,v_x,v_y))\in E \}\, \bigcup \, \{(b,(0,y,v_x,v_y))\in E\}\,
\end{equation*}
satisfies Assumption~\ref{AssImp}, though there are velocities that might not lead to any trajectory and become dead ends because of the dynamics in the two discrete modes as happens in~\eqref{eqDeadEnd}.

The map $R_{ab}\colon B_a \rightrightarrows D_b$ is defined by $R_{ab}(d)=R(d_a)$ such that $\varphi_D(R_{ab}(d))=b$ for $d_a$ in $B_a$. Analogously, $R_{ba}\colon B_b\rightrightarrows D_a$ takes the jump points in the discrete mode $b$ to the successors in the discrete mode $a$. This hybrid system includes two different dynamics: the discrete mode $a$ corresponds to the free dynamics and the discrete mode $b$ has a nonholonomic constraint given by the codistribution spanned by ${\rm d}x-{\rm d}y$.

Let $(a,(x_0,y_0,v_{x_0},v_{y_0})$ be the initial condition in $D$ such that $x_0<0$ and $v_{x_0}>0$. The motion is first free and  there are many different curves  $(x(t),y(t),v_x(t),v_y(t))$ satisfying the initial conditions.
If the trajectory hits $B_a$ at a time $t^*$, that is, $x(t^*)=0$, the jump map acts and sends the trajectory to the discrete mode $b$
with initial condition $\left(0,y(t^*),\dfrac{v_x(t^* )+v_y(t^* )}{2},  \dfrac{v_x(t^* )+v_y(t^* )}{2}\right)$ if $v_x(t^*)+v_y(t^*)\geq 0$. In other words,
\begin{equation*}
R(a,(0,y(t^*),v_x(t^*),v_y(t^*)))=\left(b,\left(0,y(t^*),\dfrac{v_x(t^* )+v_y(t^* )}{2},  \dfrac{v_x(t^* )+v_y(t^* )}{2}\right)\right)\, .
\end{equation*}

If we start in the discrete mode $b$ with initial condition $(x_0,y_0,v_{x_0},v_{x_0})$ such that $x_0>0$ and $v_{x_0}<0$, the solution curve is
\begin{equation*}
x(t)=x_0+f(t), \quad y(t)=y_0+f(t)\, \,
\end{equation*}
where $f$ satisfies $f(0)=0$ and $\left.\dfrac{\rm d}{{\rm d} t} \right|_{t=0}f(t)=v_{x_0}$.
As soon as the trajectory hits $B_b$ at a time $t^*$, that is, $x(t^*)=0$, 
the jump map acts and sends the trajectory to $(a,(0,y(t^*),v_x(t^*),v_x(t^*)))$. In other words,
\begin{equation*}
R\left(b,(0,y(t^*),v_x(t^*),v_x(t^*))\right)=(a,(0,y(t^*),v_x(t^*),v_x(t^*)))\, .
\end{equation*}

Some curves might never change the discrete mode and the trajectory would only evolve according to one generalized dynamical system, for instance, if the initial discrete mode is $b$, $(x\circ \gamma)(t)>0$ and $\dfrac{\rm d}{{\rm d} t} (x\circ \gamma)(t)>0$ for all $t$ in the domain of  $\gamma$.

\section{Global controllability of geometric hybrid control systems}\label{Sec:ControlTestHCS}

An interesting and particular case of the geometric hybrid systems consists of including controls both in the continuous and the discrete modes.
An example  of such hybrid control systems is the motion of an automobile with automatic or manual transmission. More examples
can be found in~\cite{Liberzon,Lygeros,VanDerSchaft}.
Here we restrict our attention to have controls only in the continuous dynamics.  The control set $U$ fibers onto $A$, so it could be different depending on the discrete mode we are on. Using the geometric framework in Definition~\ref{Def:GHS}, a geometric hybrid control system is just a particular case of a geometric hybrid system as described in Definition~\ref{Def:GHCS}. To emphasize the role of the controls  the global space $E$ corresponds with $F\times U$ and fibers onto $A$.

When working with control systems, controllability is one of the properties of great interest because it guarantees that between any two points on the manifold there exists a trajectory solution of the system that joins them. Controllability is a hard problem when studying classical control systems~\cite{99Agrachev}, and so is when studying hybrid control systems. However, some results about classical local controllability have already been extended to hybrid control
systems under some fairly strong assumptions~\cite{2002BulloZefran,2007YangBlanke}. In this section we provide new contributions to the geometric description of controllability for geometric hybrid control systems. At first, we obtain some global controllability tests and finish the novel results with an infinitesimal characterization using the geometry of the set of jump points and the leaves of each control system involved in the hybrid control one.

\begin{defn}\label{Def:GHCS}
A \textit{geometric hybrid control system (HCS)} is a geometric hybrid system associated with the six-tuple $(A, F\times U, M, \rho, D, R)$.
\end{defn}
If $F=M$, then $\rho=X\colon F\times U \rightarrow {TM}$ is a vector field on $M$ depending on the controls and  every generalized dynamical system is a nonlinear control system as described in Example~\ref{Ex:control}. However, Definition~\ref{Def:GHCS} also includes implicit control systems and control systems defined on Lie algebroids~\cite{Eduardo}, etc.

We will refer to a hybrid trajectory of the hybrid system in Definition~\ref{Defn:HybridTraj} as a \textit{hybrid control trajectory} to emphasize it is given by  a curve $\gamma=(\sigma,u)$ on $D\subset F\times U$.

A HCS  is \textit{discrete controllable}
if for any two discrete modes $a_1$ and $a_2$ in $A$ there exists a sequence of feasible transitions or edges in the directed graph $G$ associated with the system that goes from $a_1$ to $a_2$.
A HCS  is \textit{controllable} if for any two points $x$ and $y$ in the base manifold $M$ there exists a hybrid control trajectory that joins them. We do not consider controllability on the entire $F$. For the mechanical systems such as the ones in Example~\ref{el-eq}, once controls are added to the picture,  controllability on $F=TQ$ implies to build trajectories both on the states and the velocities as considered in~\cite{2005BulloLewisBook}.

In a local sense the controllability of a system is related with the topology of the \textit{reachable set from $x$ in $M$ at time $T$} defined as follows:
\begin{equation}
{\mathcal R}(x_0,T)=\left\{ x\in M \; \left| \, \begin{array}{l} \mbox{there exist $d_0$ in $D$ such that $(\tau_M\circ \rho)(d_0)=x_0$,}
\\  \mbox{a hybrid control trajectory $\gamma$ with initial discrete mode $\varphi_M(x_0)$} \\
\mbox{and initial condition
$d_0$ such that } (\tau_M\circ \rho \circ \gamma)(T)=x \, .  \end{array} \right. \right\}
\label{eqReachSet} \end{equation}
Note that the initial condition $x_0$ in $M$ already carries the information of the initial discrete
mode given by $\varphi_M(x_0)$ because the manifold $M$ fibers onto $A$. That restricts the trajectories since we cannot drift away from $x_0$ by using a discrete mode different from
the prefixed one, unless $\varphi_M$ is understood as a set--valued map. In such a way, for each $x_0$ in $M$, $\varphi_M(x_0)$ gives us all the feasible
discrete modes at $x_0$. From now on we consider for simplicity that $M$ is a trivial fiber bundle given by
$A\times M$ (with abuse of notation for simplicity) so that the discrete and the continuous states can be taken separately and we focus only on the reachable points on $M$.

To characterize accessibility properties for hybrid control systems, it is necessary to define the \textit{reachable set from
$x_0$ up to time $T$}:
\begin{equation*}
{\mathcal R}(x_0,\leq T)=\bigcup_{t\in (0,T]} {\mathcal R}(x_0,t)\, .
\end{equation*}

To define the following notions we consider the relative topology associated with the manifold where the continuous
dynamics takes place. A hybrid control system (HCS) is \textit{locally accessible from $x_0$ in $M$} if there exists a time $T>0$ such that
${\mathcal R}(x_0,\leq T)$ has a nonempty interior. A HCS is \textit{accessible} if it is accessible from every point $x_0$ in $M$.
 A HCS is
\textit{locally controllable (LC) from $x_0$ in $M$} if there exists a time $T>0$ such that $x_0$ lies in the interior of
${\mathcal R}(x_0,\leq T)$.
In all these definitions the fact that a point is in $M$ does not impose restrictions on the discrete mode that it belongs to.

As mentioned earlier if the hybrid trajectories are discontinuous in $M$, then the reachable set might include several disconnected subsets that could even be in different manifolds. Some of the subsets could have empty interior, but as long as there is a time so that one of them does not have empty interior, the system will be, for instance, locally accesible. In those cases the notion of neighborhood of a jump point or a successor should be conveniently adapted by using the corresponding relative topology. Some infinitesimal results for accessibility and controllability are given at the end of this Section.

To have a better idea about how the notions of accessibility and controllability differ from the intuition gained by studying classical control systems, let us consider the following two examples.

\begin{example}\label{Ex:Theoretic}
	\begin{figure}[h]\centering 
\includegraphics[scale=0.4]{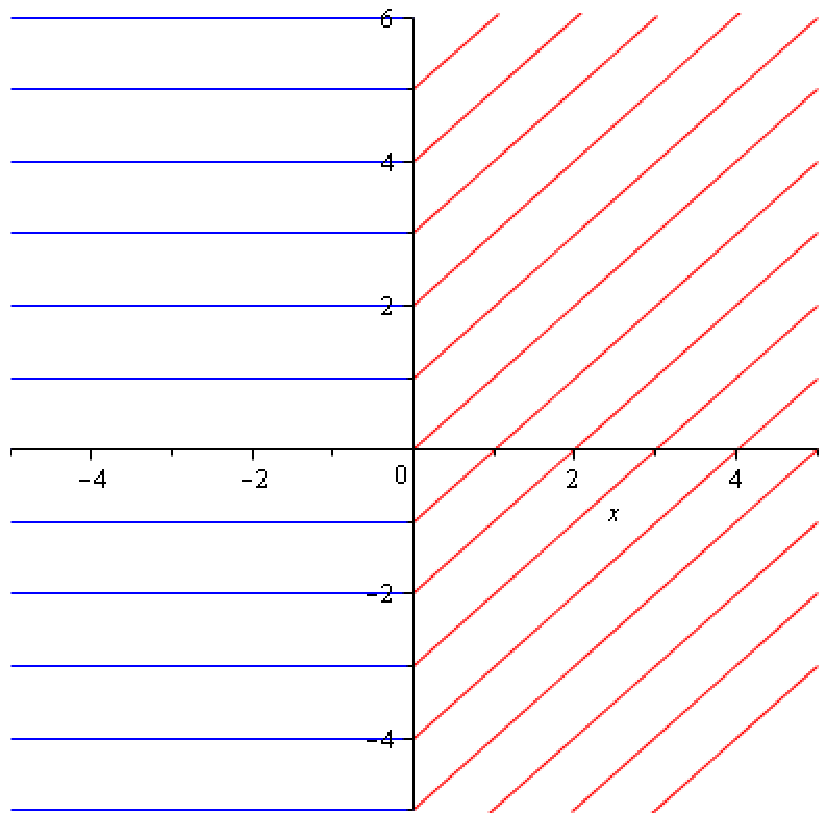} %
\hspace{5mm} \includegraphics[scale=0.4]{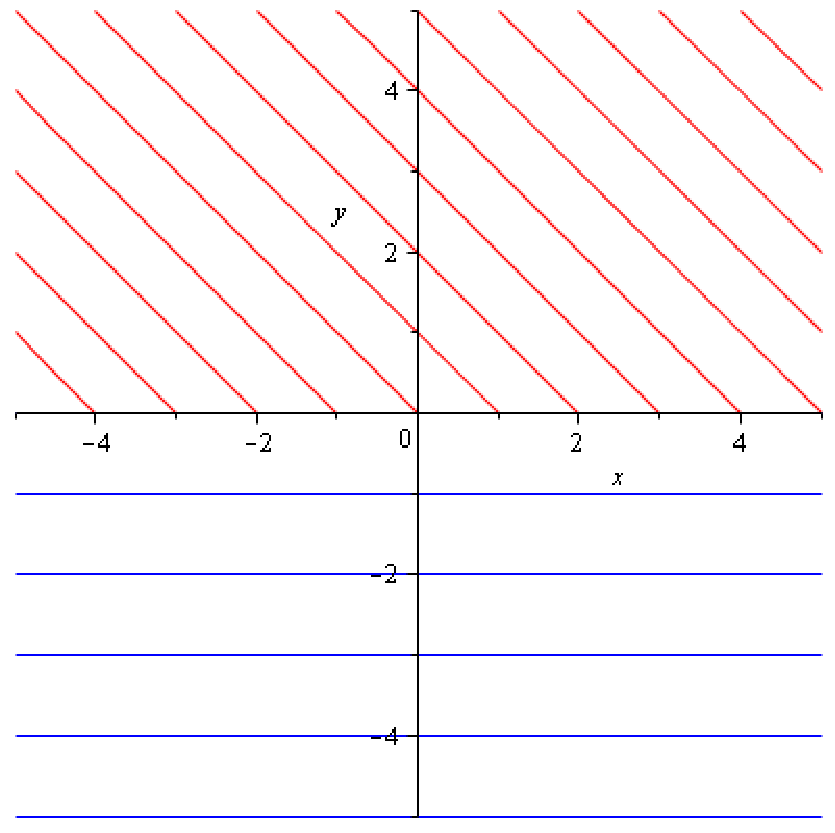}
\caption{Two examples of hybrid control trajectories.}\label{Fig:Example}
	\end{figure}

The hybrid control systems in Figure~\ref{Fig:Example} have both two discrete modes. The solutions of the systems at each discrete mode correspond to the lines in the plot that can be traced in both senses. In both cases, all the control systems are linear in the controls and the control set contains the zero in the interior. In the example on the left hand-side, the jump takes place in the $OY$--axis from one discrete mode to another and vice versa. If the jump map is not a set--valued map the systems will not be locally controllable, neither locally accessible because the reachable set will always have empty interior. It is clear from the picture that the reachable set has an empty interior. One way to gain controllability consists of defining a set--valued jump map so that every time the trajectory hits the jump set it can change to any point in the axis, that is, $R_{ab}(0,y,u)=\{(0,\widetilde{y},\widetilde{u}) \, |\,\widetilde{y}\in \mathbb{R}, \; \widetilde{u}\in U\}$. These points in the graph of $R$ are chosen conveniently so that the given initial and final points can be joined by a solution to the  hybrid control system. As a result, the system is accessible and global controllable. Note that the trajectory could be discontinuous on $\mathbb{R}^2$. The reachable set from any given point is not disconnected, but reachable sets for small time could have empty interior. Thus small-time local controllability will not be satisfied here, see~\cite{2005BulloLewisBook} for more details on that notion. 

The example on the right hand-side is not controllable, neither accessible, because the trajectories in one of the discrete modes never intersect the jump set given by the $OX$--axis. On the other discrete mode, there is no continuation after arriving at the jump set. Thus, those two modes do not complement each other well and the hybrid system becomes neither accessible nor controllable.
\end{example}

After providing the first intuition on how the jump map can be key to gain controllability for hybrid control systems, let us state some specific results. Now it is necessary to distinguish different cases to be able to provide necessary and/or sufficient conditions for controllability of HCS.
As mentioned earlier, a hybrid system can be represented by a graph with as many vertices as discrete modes and whose edges are the feasible transitions between
two discrete modes. This graph, denoted by $G$, is directed because the edges have a direction associated to it. Remember that a directed graph is
\textit{strongly connected} if it contains a directed path from $a_1$ to $a_2$ 
for every pair of vertices $a_1$ and $a_2$. When a directed edge exists in the graph, $(a_1,a_2)\in G$, it will imply that the trajectories can make the transition from mode $a_1$ to mode $a_2$.

\begin{prop}\label{Prop:DiscreteControl}
 A $\Sigma$ is a discrete controllable HCS  if and only if the graph is strongly connected.
\end{prop}
The proof is straightforward using the definitions of discrete controllability and strongly connected directed graph. For instance, the system in the right-hand side of Figure~\ref{Fig:Example} is not discrete controllable because there is no way to go from one system to another by using admissible trajectories. Hence, that edge does not exist in the graph.

\begin{prop} Let $\Sigma$ be a discrete controllable HCS. If at each discrete mode $\{a\}$ the generalized dynamical system is controllable, then the geometric hybrid control system (HCS) is global controllable.
\begin{proof}
 The assumption of discrete controllability guarantees that for any two discrete modes there exists a sequence from one to the other and vice versa. If at each discrete mode $a$ the corresponding continuous control system
is controllable, then any two points in $M_a$ can be joined by hybrid control trajectories, in particular, when one of them is in the set $\tau_E(B_a)$ of jump points. Thus any two points at different discrete modes are also joined by hybrid control trajectories because we only have
to identify the path of discrete modes that takes us from the initial discrete mode to the final one. As every generalized dynamical system $\Sigma_a$ is controllable, there exists a trajectory that will take us from the successor points to the next jump point.
\end{proof} \label{Prop:ControlConnected}
\end{prop}

\begin{remark} The inverse statement of Proposition~\ref{Prop:ControlConnected} is not necessarily true. As shown in Example~\ref{Ex:Theoretic}, there are controllable HCS whose generalized dynamical systems are not controllable by themselves. However, the discrete transitions and jump maps added to them by the hybrid nature make them controllable.
\end{remark}

Now, we state an algorithmic result to check the controllability of HCS so that any two points can be joined by hybrid
trajectories. From now  on
we focus on nonlinear
hybrid control systems where $F=M$ and the jump set only impose restrictions in the states and the controls because $B\subseteq M\times U$. At each discrete mode $a$ there is a control system as the one in Example~\ref{Ex:control} with $\rho(a,x_a,u_a)= X_a(x_a,u_a)=X_a^{u_a}(x_a)$, where $X_a^{u_a}$ is a vector field on $M_a$, that is, $X_a^{u_a}\in \mathfrak{X}(M_a)$ for every $u_a\in U_a$.
Let
$x_0\in M_a$,
all solutions to HCS from $x_0$ in the discrete mode $a$ are given by concatenations of flows 
$\phi^{X_a^{u_s}}_{t_s}\colon M_a\rightarrow M_a$,
of vector fields 
$X_a^{u_s}$
 on $M_a$ as follows
\begin{equation}\label{eq:Lx0}
 L_{x_0}=\left\{ (\phi^{X_{a}^{u_l}}_{t_l}\circ \ldots \circ \phi^{X_{a}^{u_1}}_{t_1})(x_0) \; | \; t_s\in {\mathbb R}^+, \; u_s\in U_a, \; X_a^{u_s}\in  \mathfrak{X}(M_a)
 \; \hbox{ for }
 1\leq s\leq l\right\}\subseteq M_a\, ,
 \end{equation}
 Given the initial point $x_0$ the flow $\phi^{X_a^{u_s}}_{t_s}$ travels from $x_0$ according to the dynamics $X_a^{u_s}$ for an interval of time of length $t_s$ and constant control. If the controls are not constant, but dependent on time the flow of time-dependent vector fields (see~\cite{1988AbMaRa} for instance) must be used in the above equation. 

Remember that $M_a$ could be a manifold with corners, etc.
This can be defined for each discrete mode of the HCS. Once the jump has taken place, concatenations of trajectories in the new discrete mode can be considered as appears in Equation~\eqref{Eq:CondTheoremnStates}.

\begin{thm} \label{thm:nStatesLinear}
Let $A=\{1,\dots, n\}$ and $B_a$ be the set of points of $D_a\subseteq M_a\times U_a$
 where the jumps from the discrete mode $a$ to another take place.
Assume that the HCS is discrete controllable. The system is global controllable if and only if for every pair of discrete modes $a,\; b$ in $A$ and for all points $x_a$ in $M_a$ and $x_b$ in $M_b$ there exists a sequence of
discrete modes $(a_0=a,a_1,\dots, a_{k-1},a_k=b)\in A^{k+1}$ such that $(a_l,a_{l+1})\in G$ for all $l=0,\dots,k-1$, and a sequence of jumping points $(y_{a_1},\dots, y_{a_{k}})\in B_{a_0}\times \dots  \times B_{a_{k-1}}$ such that $x_a\in \tau_E(y_a)$, $x_b\in \tau_E(y_{a_{k+1}})=\tau_E(y_b)$:
\begin{equation} \tau_E(y_{a_1})\in L_{\tau_E(y_{a})},\quad
\tau_E(y_{a_{l+1}})\in L_{\tau_E(R_{a_l \, a_{l+1}}(y_{a_l}))}=\bigcup_{y\in R_{a_l \, a_{l+1}}(y_{a_l})}L_{\tau_E(y)} \,, \label{Eq:CondTheoremnStates}
 \end{equation}
for all $l=1,\dots, k$.

\begin{proof}
The statement is an algorithmic description of the necessary and sufficient conditions so that an admissible trajectory exists between any two given points. The assumptions in the theorem and equation~\eqref{Eq:CondTheoremnStates} guarantee that at each discrete mode the next jumping point
can be reached starting from the successors of the previous jumping point. Thus it is possible to algorithmically construct a solution to HCS joining any two given points
under the assumptions in the result.
\end{proof}
\end{thm}

Theorem~\ref{thm:nStatesLinear} gives a necessary and sufficient condition for controllability. The conditions guarantee the algorithmic construction of trajectories between any two points, but it is an existence theorem. Thus for every two points the sequence of discrete modes and jumping points must be explicitly found so that the condition is satisfied. Let us give an example where
Theorem~\ref{thm:nStatesLinear} can be applied. Note that we must identify all  trajectories of HCS to determine the controllability of the system. 
In the example below we identify all possible kinds of trajectories in each mode and observe that together with the jump map we can cover the entire configuration manifold and move from one point to another by joining admissible trajectories. 

\begin{example}\label{Ex:theory}
Consider the following hybrid control system with three discrete modes $A=\{1,2,3\}$ described as follows:
\begin{itemize}
\item $M_1=\mathbb{R}^2$, $M_2=\mathbb{R}^2$, $M_3=\{(x,y)\in \mathbb{R}^2\; | \; x\geq -2\}$;
\item $E=\ds{\bigcup_{a=1}^3(M_a\times \mathbb{R})}$;
\item $U_1=[-1,1]$, $U_2=[-1,1]$, $U_3=[-1,0]$;
\item $D_a=M_a\times U_a$ for $a=1,2,3$;
\item the dynamics is given by
\begin{equation*} \rho(a,x,y,u)=\left\{ \begin{array}{lcl} (1,x,y,u,0) & {\rm if} & a=1 , \\ (2,x,y,-uy,ux) & {\rm if} & a=2 , \\ (3,x,y,u,0) & {\rm if} & a=3 . \end{array}\right. \end{equation*}
\end{itemize}
The jump map is the following one:
\begin{eqnarray*}
R\colon D&\rightrightarrows &D\\
(1,1,y,u_1)& \rightrightarrows & \{(2,1,y,u_2)\, |\, u_2\in U_2\}\\
(2,x,g(x,y,x_f,y_f),u_2)&  \rightrightarrows & \{(3,x,g(x,y,x_f,y_f),u_3)\, | \, u_3\in U_3\} \\
(3,-2,y,u_3)& \rightrightarrows & \{ (1,-2,y,u_1)\, | \, u_1\in U_1\}
\end{eqnarray*}
where the jump condition from 2 to 3 depends on the jump point $(x,y)$ in $M_2$ and the final point $(x_f,y_f)$ in $M_b$ depends on the function
\begin{equation}\label{eq:g}
g(x,y,x_f,y_f)=\left\{ \begin{array}{lcl}  \mbox{min} \{\sqrt{x^2+y^2},y_f\} & \mbox{if} &  y<y_f,\\
 \mbox{max} \{-\sqrt{x^2+y^2},y_f\} & \mbox{if} &  y\geq y_f\, .
 \end{array}\right.
\end{equation}
First of all, note that the system is discrete controllable because of the definition of $R$. Hence, Theorem~\ref{thm:nStatesLinear} can be used. In the jumps there is
freedom to choose the value of the controls because the jump map $R$ is a set--valued map. We must choose them conveniently so that
the final point can be reached. Let us construct some hybrid control trajectories of the hybrid control system under consideration.

\begin{figure}[h]\centering 
\includegraphics[scale=0.5]{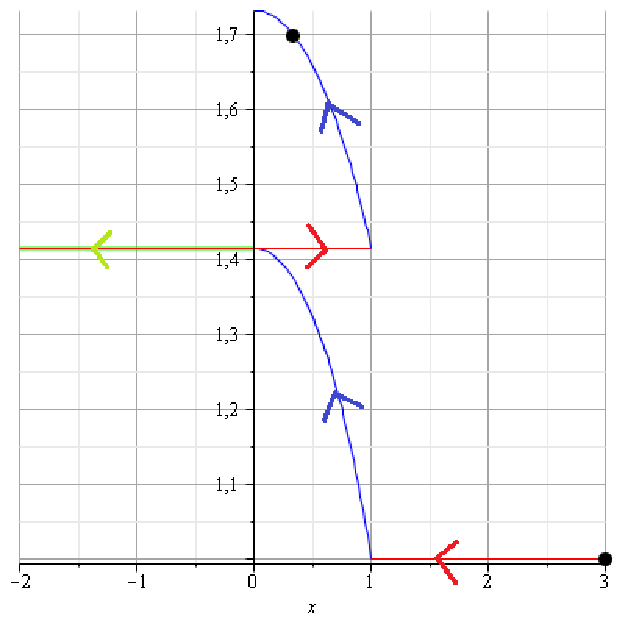} \hspace{4mm}  \includegraphics[scale=0.5]{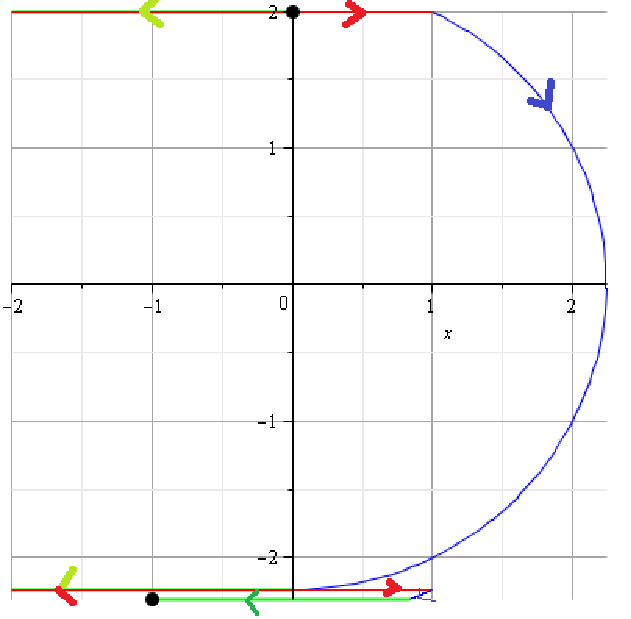}
\caption{Possible hybrid trajectories.} \label{Fig3modes}
\end{figure}

The left hand side of Figure~\ref{Fig3modes} 
shows a trajectory from $(3,1)$ starting with discrete mode 1 to $(0.332,1.697)$. The right hand side of Figure~\ref{Fig3modes}
shows a trajectory from $(0,2)$ starting with discrete mode 3 to $(-1,-2.3)$. The trajectories in red correspond with the discrete mode 1, the blue ones with the discrete mode 2, the green ones with the discrete mode 3.

The solutions of the control systems at discrete modes $1$ and $3$ are horizontal lines. At the mode $3$ the integral curves move towards
 $x=-2$ on $M_3$. At the discrete mode $2$ the trajectories with control $u$ and initial point $(\bar{x},\bar{y})$ are circles of
radius $\sqrt{\bar{x}^2+\bar{y}^2}$
 and centered at $(0,0)$. Hence the first argument in the minimum/maximum function in~\eqref{eq:g} is constant and equal to the radius. Note that the trajectories in the discrete mode 2 are the only way to reach any value for the $y$ coordinate, although more than one transition through the discrete
mode 2 might be necessary to get the final value $y_f$. In the discrete modes 1 and 3 the trajectory always evolves parallel to the $OX$--axis. 
We have identified all possible trajectories of the system and how they connect between the different modes.
It is clear from the trajectories associated with the different discrete modes that none of the discrete modes is
controllable by itself but all together make this HCS controllable because the conditions in Theorem~\ref{thm:nStatesLinear} are satisfied by construction. 
\end{example}

Let us consider some specific cases of hybrid control systems where some transversal notions are useful to characterize controllability when only two discrete modes exist and one of the systems is controllable. Associated with the jump map $R$, we introduce the set--valued map $R_{ab}\colon B_a \rightrightarrows D_b$ defined by $R_{ab}(y)=R(a,y)$ such that $\varphi_D(R(a,y))=b$ for $y$ in $B_a$. In all the proofs in the sequel Assumption~\ref{AssImp} is crucial to guarantee the existence of points to build the hybrid trajectories.

\begin{thm} \label{Prop:ControlFoliation} Let $\Sigma$ be a discrete controllable HCS such that $A=\{a,b\}$ and $G=\{(a,b),(b,a)\}$ and the nonlinear control system at the discrete mode $a$ is controllable. 
If the following conditions are satisfied for all $y_a\in B_a$
\begin{equation}\bigcup_{y\in R_{ab}(y_a)}L_{\tau_E(y)} =M_b,\label{Eq:Transverse}\end{equation}
and
\begin{equation}L_{\tau_E(y)} \bigcap \tau_E(B_b)\neq \emptyset \quad \mbox{for all } y\in D_b,\label{Eq:Transverse2}\end{equation}
then the HCS is controllable.

\begin{proof} Condition~\eqref{Eq:Transverse} guarantees that the set of all trajectories 
	having a successor from the discrete mode $a$ as initial condition 
covers the whole manifold $M_b$.
 Condition~\eqref{Eq:Transverse2} guarantees that if it is necessary to go back to the discrete mode $a$ from $b$ there exists such a trajectory in $M_b$.

The proof is obtained by construction of the trajectories given any two points $x_0, \;x_f\in M$. We need to consider four different cases:

\begin{description}

\item[Case 1]
If the initial and final conditions are in the discrete mode
$a$, by the hypothesis of controllability there is always a trajectory that joins both points within the discrete mode $a$.

\item[Case 2]  If
the initial condition $x_0$  is in $M_a$ and the final condition $x_f$ is in $M_b$, we must find a pair of points $(y_a,y_b)\in {\rm Graph}\, R_{ab}$ such that
the final condition is in $L_{\tau_E(y_b)}$.  Condition~\eqref{Eq:Transverse} guarantees that there exists a leaf in $M_b$ starting at ${\tau_E(y_b)}$ and containing the final point. Assumption~\ref{AssImp} in the jump map makes possible to find $y_a$ such that $(y_a,y_b)\in {\rm Graph}\, R_{ab}$. By the hypothesis of controllability at the mode $a$, it is always possible to find a path joining the
initial condition and $\tau_E(y_a)$. In this case, we do not need such strong conditions of the jump map. Surjectivity of $R_{ab}$ would have been enough.

 \item[Case 3] If the initial
condition $x_0$ is in $M_b$ and the final condition $x_f$ is in $M_a$, we must find a pair of points $(y_b,y_a)\in  {\rm Graph}\, R_{ba}$ such that $\tau_E(y_b)$ belongs to the intersection of $L_{x_0}\cap \tau_E(B_b)$. Condition~\eqref{Eq:Transverse2} guarantees that the intersection is not empty and $y_b$ exists because Assumption~\ref{AssImp} implies that for any point $x_b\in \tau_E(B_b)$ there exists $y_b\in B_b$ such that $\tau_E(y_b)=x_b$. The controllability hypothesis guarantees that there exists a path joining $\tau_E(y_a)$ and the final point. It does not matter the chosen $y_a$ because the system is controllable at the discrete mode $a$ in $M_a$.

\item[Case 4] If the initial and final conditions are
in the discrete mode $b$, it will be necessary to go through the discrete mode $a$ to obtain the trajectory. We divide this problem into two different ones: from the initial condition to a point in the discrete mode $a$ falls into case 3, from the mode $a$ to the final point falls into case 2.
\end{description}
\end{proof}
\end{thm}

\begin{remark} Theorem~\ref{Prop:ControlFoliation} can be extended to more than two discrete modes for some particular graphs associated with the HCS. For instance, if the graph contains a tree of height one whose root is the controllable system. Then conditions~\eqref{Eq:Transverse} and~\eqref{Eq:Transverse2} must be satisfied for all the leaves of the root, see~\cite{GraphBook} for notions coming from graph theory.
\end{remark}

\begin{example}\label{Ex:2HCS}
Consider the following hybrid control system. Let $A=\{1,2\}$, $G=\{(1,2),(2,1)\}\subseteq A\times A$,
$M_1=\mathbb{R}^{\leq 0}\times \mathbb{R}$, $M_2=\mathbb{R}^{\geq 0}\times \mathbb{R}$, $D_1=\{(x,y,w_1,w_2)\, | \, x \leq 0, \, w_1,\, w_2 \in [-1,1] \}$,
 $D_2=\{(x,y,u)\, | \, x\geq 0, \,  u \in [-1,1]\}$. Let $\xi\in D_1\cup D_2$,
 \begin{equation*}
\rho(\xi)=\left\{\begin{array}{lcl}  (1,x,y,w_1,w_2) & {\rm if} & \xi\in D_1\, ,\\
 (2,x,y,u,u) & {\rm if} & \xi  \in D_2\, .\end{array} \right.
\end{equation*}
The jump map is given by:
\begin{eqnarray*}
R\colon D&\rightrightarrows &D\\
(1,0,y,w_1,w_2)& \rightrightarrows & \{(2,0,\widetilde{y},u)\, |\,\widetilde{y}\in \mathbb{R}\, , \, u  \in [-1,1]\},\\
(2,0,y,u)&  \rightrightarrows & \{(1,0,y,\widetilde{u},\widetilde{u})\, | \,  \widetilde{u}  \in [-1,1]\} \, .
\end{eqnarray*}
 Note that at the discrete mode $1$ the control system is fully actuated and controllable, whereas at the discrete mode 2
 the control system is single input with the control vector field $\partial / \partial x+\partial / \partial y$. As a consequence the admissible velocities of the trajectories in the mode 2 will be always proportional to that single control vector field.  Thus, the jump set is given by $B=\{(1,0,y,w_1,w_2)\,| \, w_1,w_2\in[-1,1]\} \cup \{(2,0,y,u)\,|\, u\in[-1,1]\}$ and it clearly satisfies that the fibers of $B_a$ over $M_a$ have maximum dimension as required in Assumption~\ref{AssImp} and used in the cases 2 and 3 of the proof of the Theorem~\ref{Prop:ControlFoliation}.
 It can be easily proved that conditions~\eqref{Eq:Transverse} and~\eqref{Eq:Transverse2} are satisfied for this example. Thus
 the HCS is controllable by Theorem~\ref{Prop:ControlFoliation}. Note that the hybrid trajectories could be discontinuous or absolutely continuous, that is, differentiable almost everywhere.

\end{example}

 A more general sufficient condition for controllability than the one in Theorem~\ref{Prop:ControlFoliation} is the following one where none of the discrete modes need to be controllable. The proof is once again algorithmic.

\begin{thm}\label{thm:2statesLinear}
Let $\Sigma$ be a discrete controllable HCS such that $A=\{a,b\}$. Let $x_a\in M_a$, 
denote by $S_b^{L_{x_a}}$ the set of successors on $D_b$ coming from the nonempty set
$B_a\cap \tau_E^{-1}\left(L_{x_a} \right)$, that is,
\begin{equation*}S_b^{L_{x_a}}=R_{ab}\left(B_a\cap \tau_E^{-1}\left(L_{x_a}\right)\right),
\end{equation*}
where $R_{ab}\colon D_a \rightrightarrows D_b$ is a set--valued jump map from discrete mode $a$ to mode $b$.
 If for every $x_a\in M_a$ and $x_b\in M_b$ with $a\neq b$ in $A$ there exists $y\in S_b^{L_{x_a}}$ such that
\begin{equation}\label{eq:thm}
x_b\in  L_{\tau_E(y)}  \, , 
\end{equation}
then the HCS is controllable.
\begin{proof}
By construction. It follows a similar reasoning as the proof of Theorem~\ref{Prop:ControlFoliation}. If the condition~\eqref{eq:thm} holds for every $x_a\in M_a$ and $x_b\in M_b$ with $a\neq b$ in $A$, then the cases 2 and 3 in the proof of  Theorem~\ref{Prop:ControlFoliation} are proved. It only remains to study the case 1 and 4, when the initial and final conditions are in the same discrete mode
	$a$. If so, we proceed as in the case 4 by taking an intermediate point in $M_b$ whenever the points cannot be joined by continuing in the same discrete mode.

Even though none of the discrete modes are controllable, the HCS is controllable because the hypothesis guarantee the construction of a control hybrid trajectory joining any two points $x_a$ and $x_b$ .
\end{proof}
\end{thm}

From a computational perspective it is useful to rewrite the conditions for controllability in Theorem~\ref{thm:2statesLinear} in an infinitesimal way.

\begin{prop}\label{Prop:Infinit} Let $\Sigma$ be a HCS. If $\tau_E(B_a)\cap L_{x_a}\neq \emptyset $ for all the leaves $L_{x_a}$ in the regular foliation determined by $\rho(D_a)$ for every $x_a\in M_a$,
	then
	\begin{equation}\label{Condition-corol}
	T_x\tau_E(B_a)+T_x L_{x_a} = T_x M_a, \quad \forall  \; x\in \tau_E(B_a)\cap L_{x_a}\, ,
	\end{equation}
	where $L_{x_a}$ is the leaf of the foliation associated with the control system in the mode $a$.
\end{prop}

\begin{proof} It is necessary 
	for this proof some knowledge on theory of foliations~\cite{FoliationBook} and the orbit theorem~\cite{AndrewNotes}.
	Regular leaves foliate the manifold $M_a$ in disjoint submanifolds having all the same dimension. The definition of foliation implies that $\bigcup_{x_a\in M_a} L_{x_a}=M_a$.
	
Locally, we can assume that the leaves are described by $\{(x^s,x^\alpha)\, | \, x^\alpha=c^\alpha\}$ where the number of $s$-coordinates corresponds to the dimension of the leaves.

Assuming that the set $\tau_E(B_a)$ is a submanifold, it could be defined by constraints $\Phi(x^s,x^\alpha)=0\in \mathbb{R}^{n_a-{\rm dim}\tau_E(B_a)}$.

Thus, $L_{x_a}\cap \tau_E(B_a)$ is described locally by $\{(x^s,c^\alpha)\, | \, \Phi(x^s,c^\alpha)=0\}$. The assumption of the theorem guarantees that for every $c^\alpha$ there exist $x^s$ such that $\Phi(x^s,c^\alpha)=0$.
Considering $x^s$ as a function of $c^\alpha$ we can differentiate the constraints with respect to $c^\alpha$:

\begin{equation*}
\dfrac{\partial \Phi}{\partial x^\alpha}(x^s(c^\alpha),c^\alpha)=\dfrac{\partial \Phi}{\partial x^s}(x^s(c^\alpha),c^\alpha) \, \dfrac{\partial x^s}{\partial x^\alpha}(c^\alpha)+\dfrac{\partial \Phi}{\partial x^\alpha}(x^s(c^\alpha),c^\alpha)=0\, .
\end{equation*}

As $T_x\tau_E(B_a)=\ker D\Phi(x)$, at the intersection points
	\begin{equation*}
T_x\tau_E(B_a)={\rm span}\left\{\dfrac{\partial}{\partial x^\alpha}+\dfrac{\partial x^s}{\partial x^\alpha}\dfrac{\partial}{\partial x^s}\right\}\,.
	\end{equation*}
	Moreover,
	\begin{equation*}
	T_xL_{x_a}={\rm span}\left\{\dfrac{\partial}{\partial x^s}\right\}\, .
	\end{equation*}
	We can conclude the sum of the above tangent spaces spans the whole tangent space of $M_a$.
	
	It is crucial that the assumption is satisfied for all $x_a\in M_a$ so that $\tau_E(B_a)$ meets all the leaves and
	\begin{equation*}
	T_x\tau_E(B_a)+T_x L_{x_a} = T_x M_a, \quad \forall  \; x\in \tau_E(B_a)\cap L_{x_a}.
	\end{equation*}
 The reasoning still works if not all the leaves have the same dimension, but to simplify the local proof we assume that the foliation is regular.
\end{proof}

Note that the submanifolds $\tau_{E}(B_a)$ and $L_{x_a}$ are not necessarily transversal in the usual sense since the intersection of the tangent spaces could be nonempty.

Example~\ref{Ex:Theoretic} contains one case where the assumption in Proposition~\ref{Prop:Infinit} is satisfied, and hence~\eqref{Condition-corol} follows. However, in the left-hand example in Figure~\ref{Fig:Example} depending on how the jump map is defined the HCS will be controllable or not. The right-hand example in Figure~\ref{Fig:Example} has one discrete mode where the assumption in Proposition~\ref{Prop:Infinit} is not satisfied and the property~\eqref{Condition-corol} cannot be written because the intersection is empty.

Theorems~\ref{Prop:ControlFoliation} and~\ref{thm:2statesLinear} have shown that to obtain some global controllability results it is necessary some extra conditions to guarantee that after the jump is possible to find a suitable leave to connect any two points in $M$. Proposition~\ref{Prop:Infinit} guarantees that from any discrete mode it is always possible to jump to a different discrete mode, if necessary. This result is valid for any number of discrete modes, not only for two modes as in the previous results.

\begin{thm}\label{Corol:Control} Let $\Sigma$ be a discrete controllable HCS with strongly connected graph $G$. If for every $(a,b)\in G$
\begin{eqnarray}
\tau_E(B_a)\cap L_{x_a}\neq \emptyset && \mbox{for every }x_a\in M_a, \label{Extra-Assumption1} \\
	T_{x_b}(\tau_E(R_{ab}(y_a)))+T_{x_b}L_{x_b}=T_{x_b}M_b && \mbox{for every } y_a\in B_a  \mbox{ and } x_b\in \tau_E(R_{ab}(y_a)),\label{Extra-Assumption2}
	\end{eqnarray}	
	then $\Sigma$ is global controllable.
\end{thm}
\begin{proof}
	It is important that the set $\tau_E(R_{ab}(y_a))$ does not only contain one point, otherwise the tangent space is not defined. Here it is highlighted once again the need of the jump map to be a set--valued map so that controllability could be gained. As in the proofs of Theorems~\ref{Prop:ControlFoliation} and~\ref{thm:2statesLinear} we consider the different cases to prove global controllability.

	\begin{description}
		\item[Case 1] Take $x_a$ in $M_a$ and $x_b$ in $M_b$ with $a\neq b$. As the system $\Sigma$ is discrete controllable, there exists a sequence of discrete modes that will take us from $a$ to $b$.  At every discrete mode condition~\eqref{Extra-Assumption1} guarantees that from any starting point the jump set is reached. Condition~\eqref{Extra-Assumption2} guarantees that the hybrid control trajectory continues onto the new discrete mode. Thus, it is possible to construct a hybrid control trajectory from $x_a$ to $x_b$.
		\item[Case 2] Take two points in the same discrete mode, $x_a$ and $\widetilde{x_a}$. We consider an intermediate point $x_b$ at a different discrete mode, $b\neq a$, that can be reached because of Case 1. Thus, there exists a hybrid control trajectory from $x_a$ to $x_b$ and from $x_b$ to $\widetilde{x_a}$. In some cases it might not be necessary to jump to a different discrete mode if the systems $\Sigma_a$ is controllable as happens in Theorem~\ref{Prop:ControlFoliation}.
	\end{description}
\end{proof}

The controllable case in Example~\ref{Ex:Theoretic} whose leaves are plotted in Figure~\ref{Fig:Example} can be established using Theorem~\ref{Corol:Control}. Note that
Theorem~\ref{Corol:Control} contains a less restrictive condition for controllability than the one stated in Theorem~\ref{thm:nStatesLinear} where all the leaves
arising from one discrete mode must intersect with all the leaves in the following discrete mode.

In the classical control literature 
the infinitesimal characterization of controllability in terms of brackets of vector fields associated with the control-linear systems is well-known~\cite{NijmeijerVanDerSchaft}
These results are derived  from Chow-Rashevsky theorem~\cite{Chow} that states that a control-linear system is locally controllable if the
 involutive closure of the input vector fields spans the entire tangent space of the state manifold at every point,
as long as the zero control is in the interior of the control set. Similarly, a control-affine system is locally
controllable if the involutive closure of the drift vector
field and the input vector fields spans the entire tangent space of the state manifold at every point and any possible obstruction to controllability is cancelled by some vector fields in the involutive closure \cite{2005BulloLewisBook}.

Some effort to extend these results to hybrid control systems has been made in the literature, but strong assumptions such as the jump set is the entire manifold are needed to prove those results, see~\cite{2002BulloZefran}. These assumptions  make the conditions~\eqref{Extra-Assumption1} and~\eqref{Extra-Assumption2} in Theorem~\ref{Corol:Control} always true. Hence, the jump map does not play any role in deciding the controllability of the HCS. The results developed in this paper show how the geometry of the set of jump points has a crucial role in the controllability of the hybrid control systems.

If the nonlinear control systems at each discrete mode are control-linear systems or control-affine systems, the conditions~\eqref{Extra-Assumption1} and~\eqref{Extra-Assumption2} in Theorem~\ref{Corol:Control} can be rewritten infinitesimally using the well-known Orbit Theorem~\cite{S73}. We close this section by taking a control-affine system at
each discrete mode:
\[
\dot{x}_a=X_{0_a}(x_a)+u^{s_a} X_{s_a}(x_a), \quad x_a\in M_a, 	\quad (u^1,\dots,u^{k_a})\in U_a, \quad 1\leq s_a\leq k_a \leq n_a=\dim M_a, \quad a\in A\, .
\]
The control system can be described by the following affine distribution
\[
{\mathcal C}_a=X_{0_a}+\hbox{span }\{ X_{1_a},\dots, X_{k_a}\}\, .
\] These systems are called hybrid control affine systems (HCAS).
The leaves obtained from ${\mathcal C}$ are defined as in~\eqref{eq:Lx0}.

According to the notation in Section~\ref{Sec:GenHybSyst}, $D_a=M_a\times U_a$
and $\rho_a: M_a\times U_a\longrightarrow TM_a$, \\$\rho_a(x_a, u^1,\dots, u^{k_a})=(x_a, X_{0_a}(x_a)+u^{s_a} X_{s_a}(x_a))$.

Observe that in our particular case $\rho_a(D_a)={\mathcal C}_a$ if $0$ is in the interior of every control set $U_a$. The jump map of the HCS must satisfy Assumption~\ref{AssImp}. 

\remark Assume that the drift vector field vanishes identically so that the HCAS is control-linear. Let $\hbox{Lie}^{(\infty)}({\mathcal C})$ be the smallest Lie subalgebra of ${\mathfrak X}(M)$ containing ${\mathcal C}$, that is,
\begin{eqnarray*}
	\hbox{Lie}^{(\infty)}({\mathcal C})&=&\hbox{span }\left\{ [X_l, [X_{l-1}, \ldots [X_2, X_1]\ldots ]]  \hbox{ with $l\in {\mathbb N}\backslash \{0\}$}, \; X_1, \ldots, X_l \in {\mathfrak X}(M),\right.
	\\
	&&\left.   \quad \quad
	X_s(x)\in {\mathcal C}_x,  x\in M \hbox{ for all }   1\leq s \leq l\right\}\, .
\end{eqnarray*}
As $\hbox{Lie}^{(\infty)}({\mathcal C})$ is an involutive distribution, by Frobenius' Theorem the maximal integral manifolds define a generalized foliation $\mathcal{L}$ of $M$.  The notion of ``generalized" means that the maximal integral manifolds could have different dimension depending on the point.

Whenever ${\rm Lie}^{(\infty)}({\mathcal C}_a)$ is a locally finitely generated submodule of vector fields, the Orbit Theorem  guarantees that \begin{equation*}{\rm Lie}^{(\infty)}({\mathcal C}_a)_x
=T_x L_{x_a} \end{equation*}
for every $x$ in $M_a$~\cite{AndrewNotes}. The equality also holds in the analytic case.  Under those assumptions, the conditions~\eqref{Extra-Assumption1} and~\eqref{Extra-Assumption2} can be checked respectively in an infinitesimal way by computing Lie brackets as stated in the following result.

\begin{corol} Let $\Sigma$ be a discrete controllable HCAS with strongly connected graph $G$. If for every $(a,b)\in G$, 
\begin{eqnarray*}\label{Extra-Assumption}
	T_x\tau_E(B_a)+{\rm{Lie}}_{x}^{(\infty)}({\mathcal C}) = T_x M_a &&\mbox{for every} \quad x\in \tau_E(B_a)\cap L_{x_a}\, ,
\\	T_{x_b}\tau_E(R_{ab}(y_a))+{\rm{Lie}}_{x_b}^{(\infty)}({\mathcal C})=T_{x_b}M_b && \mbox{for every} \quad x_b\in \tau_E(R_{ab}(y_a))\, ,
	\end{eqnarray*}
then $\Sigma$ is locally accessible.
\end{corol}

This result can be extended to sufficient conditions for local controllability if the obstructions to controllability are neutralized similarly as written in~\cite{2002BulloZefran}.

\subsection{Example}\label{SubS:Ex}

We consider an example here to illustrate the result in Theorem~\ref{Corol:Control}. Let us consider a submarine under the sea or a sea water animal that depending on the depth can only move in one direction because of the tides. To simplify the model, we only consider that the submarine can move vertically and horizontally, but it can be easily extended to the real three--dimensional space. This model has two modes, one allows to move to the right at a particular depth and the other one allows to move to the left, as shown in Figure~\ref{Fig:Tides}.
At two particular positions of the $OX$--axis the submarine can ``instantaneously"  goes up and down as much as necessary.

 \begin{figure}[h!]\centering

\includegraphics[scale=0.6]{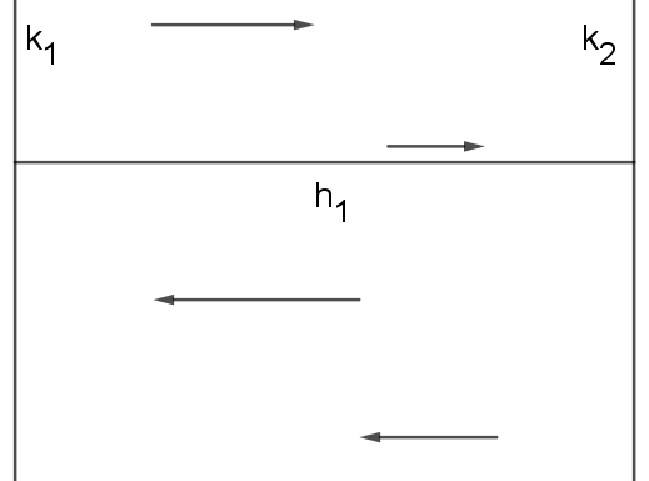}
\caption{Setting of the geometric hybrid control system.} \label{Fig:Tides}
\end{figure}

Formally speaking, the geometric hybrid control system is given by the six--tuple $(\{1,2\},M\times U,M,\rho,D,R)$, where
\begin{itemize}
\item $M_1=\{(x,y)\in \mathbb{R}^2\, | \, k_1\leq x \leq k_2, \; y> h_1\}$ for fixed numbers $h_1,\, k_1, \, k_2\in \mathbb{R}$ and $U_1=\mathbb{R}^{>0}$ only includes strictly positive numbers;
\item $M_2=\{(x,y)\in \mathbb{R}^2\, | \,  k_1\leq x \leq k_2, \; y\leq h_1\}$ for fixed numbers $h_1,\, k_1, \, k_2\in \mathbb{R}$ and $U_2=\mathbb{R}^{<0}$ only includes negative numbers;
\item $\rho_1(x,y,u)=(x,y,u,0)$ and $\rho_2(x,y,u)=(x,y,u,0)$;
\item $D=(M_1\times U_1)\cup (M_2\times U_2)$;
\item The jump map is given by:
\begin{eqnarray*}
R\colon D&\rightrightarrows &D\\
(1,k_2,y^-,u^-)& \rightrightarrows & \{(2,k_2,y^+,u^+)\, |\, (k_2,y^+,u^+)  \in M_2\times U_2\, , \, y^+\leq h_1\},\\
(2,k_1,y^-,u^-)&  \rightrightarrows & \{(1,k_1,y^+,u^+)\, | \,  (k_1,y^+,u^+) \in M_1\times U_1\, , \,  y^+>h_1\} \, .
\end{eqnarray*}
\end{itemize}

The control systems for both states are single-input linear control system on a configuration manifold of dimension 2. The two systems independently are clearly not accessible, neither controllable. However, as a geometric hybrid control system it satisfies the conditions in Theorem~\ref{Corol:Control} because the tangent space of the leaves at jump points is $\partial /\partial x$ and the tangent space of the jump set is given by $\partial /\partial y$.

It can be checked that the global controballity is obtained by using in particular the following jump map from any initial condition to the final condition $(x_f,y_f)$:
\begin{equation*}
 R(a,k_2,y,u)=\left\{ \begin{array}{lcl} (1,k_2,y_f ,-u)& \mbox{if} &x_f\in M_1 \mbox{ and } (k_2,y)\in M_2, \\
(2,k_2,h_1,-u) & \mbox{if} &x_f\in M_1 \mbox{ and } (k_2,y)\in M_1 ,\\
(1,k_2,h_1+|y|,-u) & \mbox{if} &x_f\in M_2 \mbox{ and } (k_2,y)\in M_2, \\
(2,k_2,y_f,-u) & \mbox{if} &x_f\in M_2 \mbox{ and } (k_2,y)\in M_1.
 \end{array}\right.
\end{equation*}

\section{Future work}

We have introduced a new characterization to gain controllability for hybrid control systems by exploiting the geometry of the jump sets. No similar results are known in the literature to our best knowledge. As appears along the paper, the hybrid nature of the system adds a complicate caseload. The geometric study of the trajectories of hybrid control systems developed here sets the foundations to construct geometric integration methods for hybrid control systems. Some attempts to construct geometric integration relies on introducing a discretization of the time interval called hybrid time interval. Thus it is the time and not the states that determines when to change to another discrete mode \cite{SinaSymmetry, SinaVariational}. There is a clear difficulty if the states are discretized because the constraints of the jump set could not be satisfied precisely. All that will be part of the future research lines to follow.

\section*{Acknowledgements}

The authors have been partially supported by Ministerio de Econom\'ia y Competitividad (MINECO, Spain) under grant MTM 2015-64166-C2-2P. MBL and DMdD acknowledge financial support from the Spanish  Ministerio de Econom\'ia y Competitividad , through the   research grants MTM2013-42870-P, MTM2016-76702-P and ``Severo Ochoa Programme for Centres of Excellence'' in R\&D (SEV-2015-0554). MBL has been financially supported by ``Programa propio de I+D+I de la Universidad Polit\'ecnica de Madrid: Ayudas dirigidas a j\'ovenes investigadores doctores para fortalecer sus planes de investigaci\'on". JC and SM have been financially supported by AFOSR Award FA9550-15-1-0108 and AFOSR Award FA9550-18-1-0158, respectively. MCML acknowledges the financial support from the Spanish Ministerio de Econom\'ia y Competitividad project MTM2014–54855–P, the Spanish  Ministerio de Ciencia Innovaci\'on y Universidades project  PGC2018-098265-B-C33 and from the Catalan Government project 2017–SGR–932.

\appendix

\section{Classical notion of a hybrid system}\label{App:HS}

We define here the most classical notion of hybrid systems that can be found in the literature~\cite{Liberzon,VanDerSchaft} to make clear that the geometric framework introduced in this paper also includes the classical notions.

 \begin{defn} A \textit{hybrid system} is a six-tuple $(A,{\mathcal M},F,E,{\rm Guard}, {\rm Reset})$ where
\begin{itemize}
\item $A$ is a finite set of discrete modes. They are the vertices of a graph and are usually natural numbers, that is, $A\subseteq \mathbb{N}$.
\item ${\mathcal M}$ is an $n$-dimensional manifold. It is the continuous state space of the hybrid system in which the continuous state variables $x$ take their values. There is a mapping ${\rm Dom}\colon A \rightrightarrows {\mathcal M}$ called domain such that it assigns to each discrete mode the set ${\rm Dom}(a)\subseteq {\mathcal M}$ where the continuous variables take values.
\item $F$ is a mapping $F\colon A\times {\mathcal M} \rightarrow T{\mathcal M}$ that assigns to each discrete mode $a$ in $A$ a vector field $F_a$ to determine the dynamics of the continuous state. More generally, it could be a mapping that assigns to each discrete mode a set of differential algebraic equations relating the continuous state variables with their time-derivatives.
\item $E\subseteq A\times A$ is a finite set of edges called transitions which determine the possible switchings between discrete modes.
\item
${\rm Guard}\colon E \rightrightarrows  \mathcal{M}$ is a set--valued map that assigns to each edge $(a,b)$ the subset ${\rm Guard}(a,b)$ of ${\mathcal M}$ where the continuous state variable must be to jump from the discrete mode $a$ to $b$.
\item
${\rm Reset}\colon E\times {\mathcal M}\rightrightarrows {\mathcal M}$ is a set--valued map that assigns to each edge $(a,b)$ and a point $x$ in ${\rm Guard}(a,b)$ the set of points where the continuous state $x$ jumps to.
\end{itemize}
\label{Defn:ClassicalHS}
\end{defn}

The arrow $\rightrightarrows$ in the maps Guard and Reset indicates that these maps are set--valued.
All the above elements can be summarized in a graph where the vertices are the discrete modes, the edges are the transitions between discrete modes. At each vertex $a$ in $A$ the continuous state variable takes values in ${\rm Dom}(a)$. The edge $(a,b)$ is only active if the continuous state variable lies in ${\rm Guard}(a,b)$. From a point $x$ in ${\rm Guard}(a,b)$ the edge $(a,b)$ takes the point $x$ to a point in ${\rm Reset}(a,b,x)$.

All the elements in this classical definition are included in the geometric definition of hybrid systems given in Section~\ref{Sec:GenHybSyst} because ${\rm Guard}(a,b)={\rm Dom} R_{ab}$ and ${\rm Reset}(a,b,\cdot)=R_{ab}$.

\subsection{Geometric hybrid systems}\label{Sec:GHS}

According to~\cite{GoebelPaper,GoebelBook} the classical definition of a hybrid system (HS) on $\mathbb{R}^n$ can be reduced to a
tuple $(C,F,B,G)$ where $F\colon C\subseteq \mathbb{R}^n \rightrightarrows \mathbb{R}^n$ is a set--valued map defining a differential inclusion
\begin{equation*}
 \dot{x}\in F(x), \quad x\in C,
\end{equation*}
and
$G\colon B\subseteq \mathbb{R}^n\rightrightarrows \mathbb{R}^n$ is a set--valued map defining a set--valued discrete-time systems
\begin{equation*}
 x^+\in G(x), \quad x\in B.
\end{equation*}
Here $x^+$ denotes the point where the continuous state variable jumps to.
Instead of having set--valued maps, we could just have maps $F$ and $G$ to define
\begin{itemize}
\item the flow condition: $\dot{x}=F(x)$,
\item the jump condition: $x^+=G(x)$.
\end{itemize}
In this description the discrete modes are not identified. In the following definition of hybrid systems on manifolds the discrete modes at each moment are obtained by means of a projection.

\begin{defn} A \textit{hybrid system} (HS) on a global space 
$X$ is a six-tuple $(X,A,C,F,B,G)$ where
\begin{itemize}
 \item $A$ is a finite set of discrete modes and it is a subset of $\mathbb{N}$,
\item $C$ and $B$ are subsets of $X$;
\item $\pi_C\colon C \rightarrow A$ is a fibration onto the discrete modes such that for each $a$ in $A$, $C_a$ is a differentiable
manifold of $X$;
\item $\pi_B\colon B \rightarrow A$ is a fibration onto the discrete modes such that for each $a$ in $A$, $B_a$ is a differentiable
manifold of $X$;
\item $F\colon C\rightarrow TC$ is a map such that for each $a$ in $A$, $F_a\colon C_a \rightarrow TC_a$ is a vector field on $C_a$;
\item $G\colon B \rightarrow X$ is a jump map.
\end{itemize}
\label{Defn:HSonX}
\end{defn}
Note that $C_a$ and $B_a$ could be differentiable manifolds with corners, boundaries, etc.

Hence, we could think of $C$ as $A\times M$ where $A$ corresponds to the discrete modes and
$M$ is the manifold where the continuous states evolve. The jump map could be defined from $B$ to $B \cup C$. However,
by considering $X$ instead of $B \cup C$ we accept that the trajectory might suddenly stop because it has jumped to an
isolated point where nor the flow condition neither the jump condition can be applied.

Definition~\ref{Defn:ClassicalHS} can be rewritten using the elements in Definition~\ref{Defn:HSonX} because it is only necessary to take:
\begin{itemize}
 \item $C_a={\rm Dom}(a)$, $C=\bigcup_{a\in A} \{a\}\times {\rm Dom}(a)$,
\item $B_a=\bigcup_{a' \mbox{ s.t. } (a,a')\in E} {\rm Guard} (a,a')$, $B=\bigcup_{a\in A} \{a\}\times B_a$,
\item $G(a,x)=\bigcup_{a' \mbox { s.t. } x\in {\rm Guard}(a,a')} (a',{\rm Reset}(a,a',x))$.
\end{itemize}
Note that $G$ could be a set--valued map as considered in~\cite{GoebelPaper,GoebelBook}.

\bibliographystyle{plain}
\bibliography{references}
\end{document}